\documentclass[11pt,oneside,leqno]{amsart}

\usepackage{amsxtra}
\usepackage{amsopn}
\usepackage{amsmath,amsthm,amssymb}
\usepackage{amscd}
\usepackage{color}
\usepackage{amsfonts}
\usepackage{latexsym}
\usepackage{verbatim}

\theoremstyle{plain}
\newtheorem{theorem}{Theorem}[section]
\newtheorem{lemma}[theorem]{Lemma}

\newtheorem{cor}[theorem]{Corollary}

\theoremstyle{definition}
\newtheorem{definition}[theorem]{Definition}
\newtheorem{rem}[theorem]{Remark}

\newcommand\R{{\mathbb R}}

\newcommand{\deltai}{\dfrac{\partial}{\partial z^i}}
\newcommand{\deltak}{\dfrac{\partial}{\partial z^k}}
\newcommand{\deltaj}{\dfrac{\partial}{\partial z^j}}
\newcommand{\deltal}{\dfrac{\partial}{\partial z^l}}

\newcommand{\co}{{\mathcal O}}
\newcommand{\Oomega}{\overline\Omega}

\begin{document}

\title[Poisson structures]{Poisson structures on twistor spaces of hyperk\"ahler and HKT manifolds}
\author{Gueo Grantcharov,  Lisandra Hernandez-Vazquez}
\address{Department of Mathematics and Statistics\\
Florida International University\\
11200 S.W. 8th Street\\
Miami, Florida 33199\\
USA
}
\email{lhernanv@fiu.edu, grantchg@fiu.edu }
\subjclass[2000]{53C26, 53D18, 53D05}
\keywords{hypercomplex manifold, holomorphic Poisson structure, hyperk\"ahler manifold}
\thanks{This work was supported by a grant from Simons Foundation (\#246184)}
\begin{abstract}
We characterize HKT structures in terms of a nondegenrate complex Poisson
 bivector on a hypercomplex manifold. We extend the characterization to the twistor space. After considering the flat case in detail, we show that the twistor space of a hyperk\"ahler manifold admits a holomorphic Poisson structure. We briefly mention the relation to quaternionic and hypercomlex deformations on tori and K3 surfaces.
\end{abstract}
\maketitle
\vspace{.2in}

\section{Introduction}

HKT structures (an abreviation from hyperk\"ahler with torsion) were first introduced in String Theory (see \cite{HP}) as the structures induced on the target manifolds of $(4,0)$-supersymmetric sigma models with Wess-Zumino term. From a mathematical viewpoint, compact HKT manifolds share many properties with the K\"ahler ones. They have local potential functions \cite{BS, GP} and well defined Hodge theory \cite{V}, which for spaces with $SL(n, H)$-holonomy, leads to a characterization in dimension eight similar to the topological characterization of K\"ahler compact complex surfaces \cite{GLV}. On a hypercomplex manifold $M$, an HKT structure is given by a real positive $(2,0)$-form $\Omega$, which is $\partial$ closed. This also has a description in terms of a $\partial$-closed 2-form on its twistor space \cite{BS, GP}. As is well known, when $\Omega$ is closed, the structure is  hyperk\"ahler and $\Omega$ is holomorphic symplectic. A characterization of a hyperk\"ahler structure from a twistor space perspective is given in \cite{HKLR} in terms of the existence of a twisted holomorphic 2-form that is nondegenerate on each fiber of the twistor projection.

 It is known that holomorphic symplectic forms are dual to holomorphic Poisson bivectors of maximal rank. Holomorphic Poisson structures have been studied from different perspectives. Recently, the interest in such structures is growing due to their  connection to generalized complex geometry. One purpose of this note is to find characterizations of HKT structures on hypercomplex manifolds and their twistor spaces in terms of Poisson structures.

 To this end, we need the notion of complex Poisson structures. Lichnerowicz has introduced complex Poisson structures in \cite{L}. However, the definition in \cite{L} is more restrictive than the one we need. We collect the necessary facts about complex Poisson structures in Section 2. They are straightforward analogs of the properties of the  real Poisson structures. We note in particular that a nondegenerate complex Poisson structure is dual to a $\partial$-closed nondegenerate $(2,0)$ form. A direct consequence is the characterization in Section 3 of an HKT structure in terms of a complex Poisson one. Since a characterization in terms of twistor spaces is given in \cite{BS, GP}, we formulate a similar existence result for complex Poisson structures on the twistor spaces of HKT manifolds. We also provide an explicit form for the complex Poisson structures which arise in this way on $SU(3)$ and $SU(5)$. In this way we see that there is 1-parameter family of invariant HKT structures on $SU(3)$ and they are related to the holomorphic symplectic form on the nilpotent orbit of the highest root in $SL(3, \mathbb{C})$. The complex Poisson structure on $SU(5)$ is a sum of one on some embedded $SU(3)$ and one on $SU(5)/SU(3)$. In Section 4, we consider in detail the flat case using local coordinates. In particular, we observe that the twistor space has many commuting holomorphic vector fields, so there are also many holomorphic Poisson structures. Moreover, the complex Poisson structure found in Section 3 is also holomorphic. In the Section 5 we prove that this structure is holomorphic on the twistor space of any (possibly curved) hyperk\"ahler manifold. Intuitively, this follows from the fact that the holomorphic structure on the twistor space, which is determined by the Chern connection, depends on the Levi-Civita connection on the hyperk\"ahler base, but not on its curvature. This is not true for general twistor spaces. In the last section we mention the relation of this holomorphic Poisson structure to the quaternionic and hypercomplex deformations in the case of K3 surfaces and tori, following \cite{Hitchin}. In both cases the infinitesimal deformations we describe lead to actual deformations.
\vspace{.2in}

{\bf Acknowledgements:} We are grateful to D.Kaledin, M.Verbitsky and Y.S.Poon for their comments and interest in this work.


\section{Complex Poisson structures and complex symplectic forms}

Although Poisson structures on complex and almost complex manifolds have been introduced earlier in \cite{L, CFIU}, we need slightly different terminology, adapted to our purposes. A {\it complex Poisson structure} is a bivector $P$ on a (almost) complex manifold
of type $(2,0)$, such that $[P,P]=0$. This definition is useful when the manifold is complex and implicit in \cite{P}. In this note, this definition will be applied in the hypercomplex case. Note that it is weaker than the one given by Lichnerovicz \cite{L} since it doesn't imply that the real and imaginary part of $P$ commute: $[Re P,Im P] \neq 0$.

Let $M$ be a complex manifold and $T^{p,q}, \Lambda^{p,q}$  the spaces of $(p,q)$-vectors and $(p,q)$-forms, respectively. Then $\partial = d+id^c, \overline{\partial}=d-id^c$ are the standard operators defined by $\partial|_{\Lambda^{p,q}}=\pi_{p+1,q}\circ d$ for the projection $\pi_{p+1,q}:\Lambda^{p+q+1}\rightarrow \Lambda^{p+1,q}$. We have the following simple observation, similar to the real case:

\begin{lemma}\label{L1}
If $P$ is a $(2,0)$ bivector on a complex manifold, given by $P = \dfrac{1}{2} P^{ij} \deltai \wedge \deltaj$ in local complex chart $(z_1,z_2,...z_n)$, then for $i<j<k$  $$[P,P](dz^i,dz^j,dz^k) = 2\left(P^{ih}\partial_h P^{jk} + P^{jh}\partial_h P^{ki}+P^{kh}\partial_h P^{ij}\right)$$
\end{lemma}
\begin{proof}

A direct calculation, similar to the real case gives
\begin{align*}
[P,P] &= \left[\dfrac{1}{2} P^{ij} \deltai \wedge \deltaj, \dfrac{1}{2} P^{kl} \deltak \wedge \deltal\right]\\
                                &=\dfrac{1}{4}\left(\left[P^{ij} \deltai, P^{kl} \deltak\right] \wedge \deltaj \wedge \deltal - \left[P^{ij} \deltai, \deltal\right] \wedge \deltaj \wedge \left(P^{kl}\deltak\right) \right)
                                \\&-\dfrac{1}{4}\left(\left[\deltaj, P^{kl} \deltak\right] \wedge \left(P^{ij}\deltai\right) \wedge \deltal + \left[\deltaj, \deltal\right] \wedge \left(P^{ij} \deltai\right)\wedge \left(P^{kl}\deltak\right)\right)\\
                                &=\dfrac{1}{4}\left( - P^{ij} \dfrac{\partial P^{kl}}{\partial z^i} \deltaj \wedge \deltak \wedge \deltal + P^{ij} \dfrac{\partial P^{kl}}{\partial z^j} \deltai \wedge \deltak \wedge \deltal - P^{kl} \dfrac{\partial P^{ij}}{\partial z^k} \deltal \wedge \deltai \wedge \deltaj \right) \\
                                & +\dfrac{1}{4} \left(P^{kl} \dfrac{\partial P^{ij}}{\partial z^l}  \deltak \wedge \deltai \wedge \deltaj \right)\\
                &=2\left(P^{ih}\partial_h P^{jk} + P^{jh}\partial_h P^{ki}+P^{kh}\partial_h\cdot P^{ij}\right) \deltai\wedge\deltaj\wedge\deltak
\end{align*}
where in the last line we use $i<j<k$.
From here the Lemma follows.
\end{proof}

As in the real case, $P$ defines a linear map $P: \Lambda^{1,0} \rightarrow T^{1,0}$ from the space of $(1,0)$- forms to the space of $(1,0)$ vectors at each point.
For each complex valued function $f$ , denote by $P(d f) = P(\partial f) = X_f$ its Hamiltonian vector field. Then $\{f,g\} := X_f(g) = -X_g(f)$ defines a bracket operation on complex functions, just like in the real case. From Lemma 2.1, it satisfies the Jacobi identity (or equivalently $[X_f,X_g] = X_{\{f,g\}}$ ) iff $P$ is (complex) Poisson. Note that for $f$ and $g$ holomorphic , $\{f,g\}$ is not necessarily holomorphic, so there is no analog of the notion of "symplectic foliation" in this case. The bracket $\{f,g\}$ is holomorphic only when the functions $P^{ij}$ are holomorphic and $P$ is called holomorphic Poisson in this case. As is well known, a holomorphic Poisson structure defines a holomorphic symplectic foliation.

Suppose now that the complex $(2,0)$ bivector $P$ is of maximal rank at each point. Then $P: \Lambda^{1,0} \rightarrow T^{1,0}$ is invertible and $\omega = P^{-1} : T^{1,0} \rightarrow \Lambda^{1,0}$ defines a $(2,0)$ form via $\omega(P(\partial f), P(\partial g))= P(\partial f, \partial g)$. Once more similar to the real case, we have:

\begin{lemma} For nondegenerate complex Poison structure $P$ and $\omega = P^{-1}$ as above
$$\partial\omega(X_f,X_g,X_h) = \frac{3}{2}[P,P](df, dg, dh)$$

\end{lemma}

\begin{proof}

\begin{align*}
d\omega (X_f,\,X_g,\,X_h) &= X_f(\omega(X_g,\,X_h)) + X_g(\omega(X_h,X_f)) + X_h(\omega(X_f,X_g))\\
&- \omega([X_f,X_g],X_h) - \omega([X_g,X_h],X_f) - \omega([X_h,X_f],X_g)\\
                &=X_f(\{g,h\}) + X_g(\{h,f\}) + X_h(\{f,g\}) + [X_f,X_g](h) + [X_g,X_h](f)+[X_h,X_f](g)\\
&= 3\left(\{f,\{g,h\}\}+\{g.\{h,f\}\}+\{h,\{f,g\}\}\right)
\end{align*}
Since $X_f,X_g,X_h$ are (1,0) vectors, $d\omega(X_f,X_g,X_h) = \partial\omega(X_f,X_g,X_h)$.
Now applying Lemma \ref{L1} we have
$$[P,P](dz_i,dz_j,dz_k) =  2\left(\{z_i,\{z_j,z_k\}\}+\{z_j,\{z_k,z_i\}\}+\{z_k,\{z_i,z_j\}\}\right)
$$
So
$$d\omega(X_{z_i},X_{z_j}, X_{z_k}) = \frac{3}{2} [P,P](dz_i,dz_j,dz_k)
$$

and the statement follows.
\end{proof}

From here we obtain:

\begin{theorem}\label{L3}
A complex non-degenerate $(2,0)$ bivector is Poisson iff its dual 2-form $\omega=P^{-1}$ is $\partial$-closed, that is, $\partial\omega=0$.

\end{theorem}

\section{Poisson structures on HKT manifolds and their twistor space}

We begin by establishing some notations in the presence of a metric. Suppose $g$ is a Riemannian metric and $I$ a complex structure on a manifold $M$ such that $\omega(X,Y)=g(IX,Y)$  is the  fundamental form. Let $\sharp: T^* \rightarrow T$ be the isomorphism given by $g(\sharp \alpha, Y)=\alpha(Y)$ and denote by
$\omega$ and $\omega^{-1}$ the maps $\omega:T\rightarrow T^*$ and $\omega^{-1}:T^*\rightarrow T$ given as $\omega(\omega^{-1}(\alpha),Y)=\alpha(Y)$ and $\omega(X) = i_X\omega$. Then $\omega(\omega^{-1})=Id|_{T^*}, \omega^{-1}(\omega)=Id|_{T}$. Thus $\alpha(Y)=\omega(\omega^{-1}(\alpha),Y)=g(I\omega^{-1}(\alpha),Y)=g(\sharp\alpha,Y)$, and $\omega^{-1} = -I\circ \sharp$. Extending $I$ on $T^*$ in a standard way, $\omega^{-1}=\sharp \circ I$. The map $\omega^{-1}$ defines a bivector  $\omega^{-1}\in \Lambda^2T$ as $\omega^{-1}(\alpha, \beta)= \beta(\omega^{-1}(\alpha))$ and it is the same as the map defined by the Poisson bivector $P$ above.

Recall that a hypercomplex structure is a triple $I,J,K$ of complex structures satisfying the quaternionic identities $IJ=-JI=K$.
Let $I,J,K,g$ be  a hyperhermitian structure and $\omega_I,\omega_J,\omega_K$ be the corresponding 2-forms. From above
$$ \sharp = I\omega^{-1}_I  = J\omega^{-1}_J = K\omega^{-1}_K $$ and from here $$I \omega^{-1}_J = \omega^{-1}_K$$
We want to consider the "complexified" version of the Poisson and symplectic structures on a Hermitian manifold. Let $\Omega = \omega_J+i\omega_K$ and denote by $\Omega: T^{1,0}\rightarrow T^{*(1,0)}$ the map given by $\Omega(X^{1,0})= i_{X^{1,0}}\Omega$. Since $\Omega|_{T^{0,1}} = 0$, it can be extended to a map on the whole complexified tangent space. We want to find the real and imaginary parts of $\Omega^{-1}|_{\Lambda^{1,0}}$. First we calculate $\omega_J\omega^{-1}_K$ via  $\omega_J\omega^{-1}_K(\alpha)(Y) = -JK\alpha(Y) = -I\alpha(Y)$. We see that $\omega_J\omega^{-1}_K = - \omega_K\omega^{-1}_J$. Then we have $(\omega_J+i\omega_K)(\omega^{-1}_J-i\omega^{-1}_K) (\alpha) = 2\alpha+2iI\alpha = 2\alpha^{1,0}$. Thus, on $\Lambda^{1,0}$ we have $\Omega^{-1}= \frac{1}{4}(\omega^{-1}_J-i\omega^{-1}_K)$ and $\omega^{-1}_J-i\omega^{-1}_K |_{\Lambda^{0,1}} = 0$, so the complex bivector $P$ given by

\begin{equation}\label{def1}
P(\partial f, \partial g) = \Omega(\Omega^{-1}(\partial f), \Omega^{-1}(\partial g))
\end{equation}
satisfies $P=\frac{1}{4}(\omega^{-1}_J-i\omega^{-1}_K)$, where $\partial$ is the operator defined by $I$.

 A hyperhermitian metric $g$ is called HKT if $d^c_I\omega_I=d^c_J\omega_J=d^c_K\omega_K$  where $d^c_I,d^c_j,d^c_K$ are the imaginary parts of the $\partial$ operators for $I,J,K$. The definition is equivalent to the existence of a hyperhermitian connection with skew-symmetric torsion, which is the original one proposed by \cite{HP}. In \cite{GP}, it is shown that the HKT condition is equivalent to: $$
\partial(\omega_J+i\omega_K)=\partial\Omega=0$$
As a direct consequence of Theorem \ref{L3} we obtain:

\begin{theorem}\label{L1}
Let $(M,I,J,K,g)$ be a hyperhermitian manifold and $P$ as in (\ref{def1}). Then $g$ is HKT iff $P$ is complex Poisson.

\end{theorem}

Let $Z = M\times S^2$ be endowed with the "tautological" complex structure $\mathcal{J}$ defined as $\mathcal{I}_{(x,{\bf a})} = (I_{\bf a}, I_{S^2})$ where for ${\bf a} = (a,b,c)$, $I_{\bf a} = aI + b J + cK$. In these terms, $I_{S^2}$ is the canonical complex structure on $S^2\cong {\mathbb CP}^1$. It is well known that the structure $\mathcal{I}$ is integrable and the complex manifold $(Z,\mathcal{I})$ is called the {\it twistor space} of $(M, I, J, K)$. The identification of  $S^2$ with ${\mathbb C}P^1$ is given by the stereographic projection:
$$St:\lambda\rightarrow {\bf a} = \left (\frac{|\lambda|^2-1}{1+|\lambda|^2}, \frac{i(\overline{\lambda}-\lambda)}{1+|\lambda|^2}, \frac{-(\lambda+\overline{\lambda})}{1+|\lambda|^2} \right)\in S^2.$$
If $\lambda$ corresponds to ${\bf a} = (a,b,c)$ via the inverse map $\lambda = St^{-1}(a,b,c)$, then $\lambda = \infty$ corresponds to $I_{(1,0,0)} = I$, $\lambda = i$ corresponds to $I_{(0,1,0)} = J$, and $\lambda = -1$ corresponds to $I_{(0,0,1)} = K$.

\begin{rem}\label{L1gen}
The integrability of $\mathcal{I}$ is equivalent to the fact that any structure $I_{\bf a}$ in the $S^2$-family $aI+bJ+cK$, for ${\bf a}=(a,b,c)\in S^2$, is integrable. Moreover for a hyperhermitian metric $g$, $I_{\bf a}$ is Hermitian. Therefore, instead of using $I,J,K$ in Theorem \ref{L1} we may as well use $I_{\bf{a}}, I_{\bf{b}},I_{\bf{c}}$ for $\bf{a},\bf{b},\bf{c}$ an orthonormal triple in $S^2$ with $\bf{c}=\bf{a}\times\bf{b}$.
\end{rem}

Suppose now that $g$ is a hyperhermitian metric on $M$ and $(Z_1, W_1, Z_2, W_2,..., Z_n, W_n)$ is a unitary basis of the space $T_x^{1,0}(M)$ of $(1,0)$-vectors for the structure $I$ at a point $x$  such that $J(Z_i)=\overline{W_i}, J(W_i) = -\overline{Z_i}$. We call such basis {\it quaternionic - Hermitian}.  For the structure $I_{\bf a}$ via the stereographic projection, the vectors:
$$ Z_i^{\lambda} = \frac{1}{\sqrt{1+|\lambda|^2}}(\overline{\lambda}Z_i  - \overline{W_i}), W_i^{\lambda} = \frac{1}{\sqrt{1+|\lambda|^2}}(\overline{\lambda}W_i  + \overline{Z_i})$$
form a unitary basis for $(g,I_{\bf a})$. Let $(\delta_i, \sigma_i),  i=1,...,n$ be the dual quaternionic-Hermitian basis of the cotangent $(1,0)$-bundle at $x$ of $(Z_i^{\lambda}, W_i^{\lambda})$. Then it is  is given by $$\sigma_i^{\lambda} = \frac{1}{\sqrt{1+|\lambda|^2}}(\lambda \sigma_i  - \overline{\delta_i}), \delta_i^{\lambda} = \frac{1}{\sqrt{1+|\lambda|^2}}(\lambda \delta_i  + \overline{\sigma_i})$$
In these terms we have that $F_{\bf a} = aF_I + bF_J  + cF_K$ is also given by $F_{\bf a} = i/2\sum (\sigma_i^{
\lambda}\wedge\overline{\sigma_i^{\lambda}} + \delta_i^{
\lambda}\wedge\overline{\delta_i^{\lambda}})$ and the dual 2-vector is given by $P_{\bf a} = -2i\sum(Z_i^{
\lambda}\wedge\overline{Z_i^{\lambda}} + W_i^{
\lambda}\wedge\overline{W_i^{\lambda}})$. We want to calculate the $(2,0)$-part $P_I|^{(2,0)}_{\bf a}$ of
$$P_I = -2i\sum(Z_i\wedge\overline{Z_i} + W_i\wedge\overline{W_i})
$$
with respect to $I_{\bf a}$. First, we see that
$$Z_i =  \frac{1}{\sqrt{1+|\lambda|^2}}(\lambda Z_i^{\lambda}  + \overline{W_i^{\lambda}}), W_i = \frac{1}{\sqrt{1+|\lambda|^2}}(\lambda W_i^{\lambda}  - \overline{Z_i^{\lambda}})$$ and similar expressions hold for the conjugates $\overline{Z_i}, \overline{W_i}$.

Now we have
$$ P_I = \frac{-2i}{1+|\lambda|^2}\sum (\lambda Z_i^{\lambda}  + \overline{W_i^{\lambda}})\wedge(\overline{\lambda} \overline{Z_i^{\lambda}} + W_i^{\lambda})  + (\lambda W_i^{\lambda}  - \overline{Z_i^{\lambda}})\wedge(\overline{\lambda}\overline{ W_i^{\lambda}}  -  Z_i^{\lambda})
$$
$$ = \frac{-2i}{1+|\lambda|^2}\sum (|\lambda|^2 - 1)(Z_i^{
\lambda}\wedge\overline{Z_i^{\lambda}} + W_i^{
\lambda}\wedge\overline{W_i^{\lambda}}) + 2\lambda Z_i^{\lambda}\wedge W_i^{\lambda} - 2\overline{\lambda}\overline{Z_i^{\lambda}}\wedge \overline{W_i^{\lambda}}
$$

So $$P_I|^{(2,0)}_{\bf a} = \frac{-4i\lambda}{1+|\lambda|^2}Z_i^{\lambda}\wedge W_i^{\lambda}$$

We can also consider ${\bf b} = \frac{1}{1+|\lambda|^2}(i(\overline{\lambda}-\lambda), 1+\frac{1}{2}(\lambda^2+\overline{\lambda}^2), \frac{i}{2}(\lambda^2-\overline{\lambda}^2))$ and ${\bf c} = \frac{1}{1+|\lambda|^2}(-(\overline{\lambda}+\lambda), \frac{i}{2}(\lambda^2-\overline{\lambda}^2), -1+\frac{1}{2}(\lambda^2+\overline{\lambda}^2) )$ and notice that $F_{\bf b} + iF_{\bf c} = \sum_i \sigma_i^{\lambda}\wedge\delta_i^{\lambda}$ as well as ${\bf c} = {\bf a}\times{\bf b}$
with ${\bf b}$ orthogonal to ${\bf a}$. Specifically, the matrix $A = ({\bf a}, {\bf b}, {\bf c})$ is a special orthogonal matrix whose inverse is its transpose. Then for the corresponding structures $I_{\bf b}$ and $I_{\bf a}$ we can find that $P_{\bf b} + iP_{\bf c} = \sum Z_i^{\lambda}\wedge W_i^{\lambda}$. Thus we have the following:

\begin{equation}\label{cxPoisson}
P_I|^{(2,0)}_{\bf a} = \frac{-4i\lambda}{1+|\lambda|^2}(P_{\bf b} +iP_{\bf c})
\end{equation}

Denote also by $\pi: Z\rightarrow S^2$ and $\pi_1:Z\rightarrow M$ the two projections - $\pi$ is holomorphic , but $\pi_1$ is not. Let also $s_{\bf a}:M\rightarrow Z$ be the section of $\pi$ defined as $s_{\bf a}(x) = (x,{\bf a})$ for every $x$. Then we can formulate the following:

\begin{theorem}\label{L4}
Let $g$ be a hyperhermitian metric on $M$ and $P_I$ be the bivector dual to the fundamental form for $I$. Let $\mathcal{P}_{(x,{\bf a})} = ((s_{\bf a})_*P_I)^{(2,0)}$ be the bivector on $Z$ defined  as the $(2,0)$ component of $(s_{\bf a})_*(P_I)$ with respect to $\mathcal{I}$.  Then $g$ is HKT iff $\mathcal{P}$ is complex Poisson.
\end{theorem}

{\it Proof:}
  First we notice that $$((s_{\bf a})_*P)^{(2,0)} = ((s_{\bf a})_*(P_I|^{(2,0)}_{\bf a})$$ where $P_I|^{(2,0)}_{\bf a}$ is as above the $(2,0)$ part of $P_I$ with respect to $I_{\bf a}$. This follows from the definition of $\mathcal{I}$ at a point $(x,{\bf a})$. From (\ref{cxPoisson}) we have  $P_I|^{(2,0)}_{\bf a} = \frac{-4i\lambda}{1+|\lambda|^2} (P_{\bf b} + iP_{\bf c})$. Then Theorem \ref{L1} and Remark \ref{L1gen} give that $[P_{\bf b} + iP_{\bf c},P_{\bf b} + iP_{\bf c}] = 0$ iff $g$ is HKT. We notice that $[(s_{\bf a})_*P|^{(2,0)}_{\bf a}, (s_{\bf a})_*P|^{(2,0)}_{\bf a}] = (s_{\bf a})_*[P|^{(2,0)}_{\bf a},P|^{(2,0)}_{\bf a}]$. Therefore, the theorem follows because $\frac{-4i\lambda}{1+|\lambda|^2}\neq 0$ for almost all ${\bf a}$, and by continuity $[\mathcal{P},\mathcal{P}]$ vanishes everywhere.

{\it Q.E.D.}

 \begin{rem} The statement is dual to the characterization of the HKT structure by Banos and Swann (\cite{BS}) in terms of the twistor space: If $F$ is the fundamental form for the Hermitian structure  $(g,I)$ then $g$ is HKT iff $\partial(\pi^*F)^{(2,0)} = 0$ where $\partial$ is the operator with respect to ${\mathcal I}$ on $Z$. Banos and Swann used it to prove that every HKT structure has a local HKT-potential. In \cite{GP}, there was a slightly different twistor characterization of the HKT condition, which is based the characterization of the hyperk\"ahler structures in \cite{HKLR}.
\end{rem}

 {\it Examples.} We explicitly describe the complex Poisson structures on the groups $SU(3)$ and $SU(5)$ corresponding to the HKT structures for the Joyce's invariant hypercomplex structures.

For a compact Lie group which is a product of a simple group and a torus of appropriate dimension, the construction is given in \cite{SSTV} and \cite{Joyce2}. It could be extended to some homogeneous spaces as well. Recently it was proven in \cite{VaskoGosho} that all invariant hypercomplex structures arise in this way, and in \cite{BGP} it was proven under the additional restriction of compatibility with the bi-invariant metric. We present here a short description following \cite{VaskoGosho}. For a compact Lie algebra $\overline{\mathfrak{g}}$ with semisimple part $\mathfrak{g}$, fix a Cartan subalgebra $\mathfrak{h}$ of $\mathfrak{g}^{\mathbb{C}}$. Choose a set of positive roots $R$ and a basis $\Pi$ of $R$.
The following definition is from \cite{VaskoGosho}:
\begin{definition}
For any $\gamma\in R$ define
$$\Phi_{\gamma} = \{\alpha\in R| \gamma-\alpha\in R \}.$$
A subset $\Gamma\subset R$ is called a stem of $R$ if
$$R = \Gamma\cup\bigcup_{\gamma\in\Gamma}\Phi_{\gamma}$$ as disjoint union.

\end{definition}

 The sets $\Phi_{\gamma}$ define a maximal strongly orthogonal set of root subsystems of $R$. The hypercomplex structure is defined in the following way. Let $\Gamma=(\theta_1,\theta_2,...,\theta_n)$ with $\theta_1$ being the highest root. First the complex structure $I$ acts as $+i$ on the positive roots and as $-i$ on the negative roots. The Cartan subalgebra $\mathfrak{h}$, possibly enhanced by an abelian ideal in $\overline{\mathfrak{g}}$, is divided in two: the space $ H = \text{Span}(H_{\theta_i})$ and orthogonal complement $H^{\bot}=\text{Span}(H^{\bot}_i)$ of the same dimension. Then $I$ interchanges $H_{\theta_i}$ and $H^{\bot}_i$. For any positive root $\alpha\in\Phi_k$ and appropriate choice of $H^{\bot}_k$, the structure $J$ acts as $J(H_{\theta_k}-iH^{\bot}_k) = E_{-\theta_k}$ and $J(E_{\theta_k-\alpha}) = c_{\alpha,\theta_k}E_{-\alpha}$  for some constant $c_{\alpha}$. We also normalize the root vectors and the elements of the Cartan subalgebra according to the Killing form. We explicitly describe the construction in the case where $\mathfrak{g}^{\mathbb{C}} = sl(2n+1, \mathbb{C}), n=1,2$.

If $E_{i,j}$ are the matrices with entry $1$ at the $(i,j)$-th place, then $[E_{i,j},E_{j,k}] =E_{i,k} = -[E_{j,k},E_{i,j}]$ and all other brackets vanish. For $n=1$, there is only one element  $\theta_1$ in the stem and $H_{\theta_1} = i(E_{1,1}-E_{3,3})$ up to a constant. Choose $H_1^{\bot}=ia(E_{1,1}+E_{3,3}-2E_{2,2})$ so the structure $I$ is defined as
$$I(E_{i,j}) = iE_{i,j}, i<j$$
$$I(E_{i,j}) = -iE_{i,j}, i>j$$
$$I(ia(E_{1,1}+E_{3,3}-2E_{2,2})) = i(E_{1,1}-E_{3,3})$$
 Here for compatibility with the biinvariant metric we should take $a=\frac{1}{\sqrt{3}}$ so that all vectors will have length $\sqrt{2}$.  For $b=1+ia$ we have:
$I(bE_{1,1}-\overline{b}E_{3,3}-(b-\overline{b})E_{2,2}) = i((bE_{1,1}-\overline{b}E_{3,3}-(b-\overline{b})E_{2,2})$.
Note that
$$[(bE_{1,1}-\overline{b}E_{3,3}-(b-\overline{b})E_{2,2})\wedge E_{1,3}, E_{1,2}\wedge E_{2,3}] =  2E_{1,2}\wedge E_{1,3}\wedge E_{2,3}$$
and
$$[E_{1,2}\wedge E_{2,3},E_{1,2}\wedge E_{2,3}] = -2E_{1,2}\wedge E_{1,3}\wedge E_{2,3}$$
A unitary basis of vectors of length $2$ in the biinvariant metric is $(i/\sqrt{3}(E_{1,1}+E_{3,3}-2E_{2,2}) + (E_{1,1}-E_{3,3}), 2 E_{i,j})$. If we define $J(i /\sqrt{3}(E_{1,1}+E_{3,3}-2E_{2,2})+ (E_{1,1}-E_{3,3}) ) = 2 E_{3,1}$ and
$J(E_{1,2}) = E_{3,2}$ the complex bivector $P$ is given by,
$$2P= (bE_{1,1}-\overline{b}E_{3,3}-(b-\overline{b})E_{2,2})\wedge E_{1,3} + 2 E_{1,2}\wedge E_{2,3}$$ where $b= 1+\frac{i}{\sqrt{3}}$.
From the calculation above, $$[P,P] = 0$$ is valid for every $b = 1+ia$. In particular, we obtain a 1-parameter family of left invariant HKT structures. One element of this family is given by the biinvariant metric and $a=\frac{1}{\sqrt{3}}$.

When $a=0$ and $b=1$ the structure is not complex Poisson, but has another interpretation. If we consider the complex Lie group $SL(3,\mathbb{C})$ then $E_{i,j}, i\neq j$ and $E_{i,i}-E_{j,j}$ are holomorphic vector fields on it with respect to its canonical complex structure. From the same calculations it follows that  $P = (E_{1,1} - E_{3,3})\wedge E_{1,3} + 2 E_{1,2}\wedge E_{2,3}$ is a holomorphic Poisson bivector on $SL(3,\mathbb{C})$. Its symplectic foliation has leaves with tangent spaces spanned by $(E_{1,1} - E_{3,3}), E_{1,3}, E_{1,2}, E_{2,3}$, Moreover, the leaf at $E_{1,3}$ could be identified with the nilpotent orbit of $E_{1,3}$. On this leaf, $P$ is dual to the canonical holomorphic symplectic form defined by the Killing form on it.

The expression for $SU(5)$ is similar. Use the basis (see \cite{VaskoGosho}) $\{ E_{i,j}, i\neq j, H_{\theta_1}=i(E_{1,1}-E_{5,5}), H_{\theta_2} = i(E_{2,2}-E_{4,4}), H_1^{\bot} = ia(E_{1,1}+E_{5,5}-2E_{3,3}), H_2^{\bot}=ia(E_{2,2}+E_{4,4}-2E_{3,3}) \}$ and define the structures $I$ and $J$ in the same way. Let
$$ P_1 = (bE_{1,1}-\overline{b}E_{5,5}-(b-\overline{b})E_{3,3})\wedge E_{1,5}+2E_{1,2}\wedge E_{2,5}+2E_{1,3}\wedge E_{3,5}+2E_{1,4}\wedge E_{4,5}
$$ and

$$P_2 = (bE_{2,2}-\overline{b}E_{4,4}-(b-\overline{b})E_{3,3})\wedge E_{2,4}+2E_{2,3}\wedge E_{3,4}$$

Straightforward calculations give $[P_1.P_1]=[P_2,P_2]=[P_1,P_2]=0$, so $P = P_1+P_2$ provides a non-degenerate complex Poisson bivector and hence an HKT structure on $SU(5)$. Also, $P_2$ is the complex Poisson structure on $SU(3)$ "centrally" embedded in $SU(5)$ from above, while $P_1$ corresponds to the HKT structure induced on the homogeneous space $SU(5)/SU(3)$.

\section{Poisson structures on the twistor space of a flat hyperk\"ahler space}
In this section we consider in detail the flat hyperk\"ahler case and see that the resulting structure on the twistor space is holomorphic Poisson. This is due to the fact that the twistor space has many commuting holomorphic vector fields. We use the set up and notations from \cite{GPP}. Choose linear coordinates $(z_1^a, z_2^a), a=1, \dots, m,$
  for $\mathbb{C}^{2m}=\mathbb{R}^{4m}$, related to the real coordinates by
$$
  z_1^a=x_{2a-1}+ix_{2a},
  \hspace{.2in}
  z_2^a=y_{2a-1}+iy_{2a}
$$

  The twistor space $Z=Z(\mathbb{R}^{4m})$ of $\mathbb{R}^{4m}$ is the bundle
  $\mathbb{C}^{2m}\otimes \mathcal{O}(1)$ on $\mathbb{C}P^1$
  \cite[Example 13.64 and Example 13.66]{Bes}.
On  $\mathbb{C}P^1$, the homogeneous coordinates are given as $[\lambda_1,\lambda_2]$ and on the open $U_i$ given by $\lambda_i\neq 0$, i=1,2, hence we have local coordinates $\lambda = \frac{\lambda_1}{\lambda_2}$ and $\tilde{\lambda}=\frac{\lambda_2}{\lambda_1}$ respectively.

On $\R^{4m} \times U_2$, the product coordinates $\{z_1^a, z_2^a, \lambda\}$ are not holomorphic.
  The holomorphic coordinates are
  \begin{equation}\label{holomorphic coordinate}
  w_1^a=\lambda z_1^a-\overline{z}^a_2,
  \hspace{.2in}
  w_2^a=\lambda z_2^a+\overline{z}^a_1,
  \hspace{.2in}
  \zeta=\lambda.
  \end{equation}
  The inverse coordinate change is
  \begin{equation}
  z_1^a=\frac{1}{1+|\zeta |^2}\left( {\overline\zeta} w_1^a+\overline{w}^a_2\right),
  \hspace{.2in}
  z_2^a=\frac{1}{1+|\zeta |^2}\left( -\overline{w}^a_1+{\overline\zeta} w_2^a\right),
  \hspace{.2in}
  \lambda=\zeta.
  \end{equation}
  In particular,
  \begin{equation}
  \frac{\partial}{\partial w_1^a}=\frac{\overline{\lambda}}{1+|\lambda|^2}\frac{\partial}{\partial z_1^a} - \frac{1}{1+|\lambda |^2}\frac{\partial}{\partial\overline{z}_2^a},
  \hspace{.2in}
   \frac{\partial}{\partial w_2^a}=\frac{\overline{\lambda}}{1+|\lambda|^2}\frac{\partial}{\partial z_2^a} + \frac{1}{1+|\lambda |^2}\frac{\partial}{\partial\overline{z}_1^a}
  \end{equation}
  are local holomorphic vector fields on $Z$ defined whenever $\lambda\neq\infty$. We notice also that
\begin{equation}
\frac{\partial}{\partial\overline{\zeta}}= \sum_a \left( -\frac{\overline{z_2^a}}{(1+|\lambda|^2)^2}(\frac{\partial}{\partial z_1^a}+\lambda\frac{\partial}{\partial\overline{z}_2^a}) - \frac{\overline{z_1^a}}{(1+|\lambda|^2)^2}(\frac{\partial}{\partial z_2^a}-\lambda\frac{\partial}{\partial\overline{z}_1^a})\right) + \frac{\partial}{\partial\overline{\lambda}}
\end{equation}
Then $\frac{\partial}{\partial\overline{\lambda}} = \frac{\partial}{\partial\overline{\zeta}} + \sum_e\left( \overline{z}_2^a\frac{\partial}{\partial\overline{w}_2^a} + \overline{z}_1^a\frac{\partial}{\partial\overline{w}_1^a}\right)$ and we see that $\frac{\partial}{\partial\lambda}$ is a smooth $(1,0)$-vector field but not holomorphic on $Z$.

  When one changes coordinates from $\lambda_2\neq 0$ to
  $\lambda_1\neq 0$, $\lambda{\tilde w}^a_j=w^a_j.$
  Therefore,
  \begin{equation}\label{vaj}
  V^a_j =\frac{1}{\lambda_2}\frac{\partial}{\partial w^a_j}
  =\frac{1}{\lambda_1}\frac{\partial}{\partial {\tilde w}^a_j}
  \end{equation}
  are globally defined as vector fields on $\mathbb{R}^{4m}\times\{\mathbb{C}^2 -(0,0)\}$.
  The vector fields on $Z$
   \begin{equation}\label{twisted fields}
  W_k^a=\frac{1}{2}\left(I_k\frac{\partial}{\partial x_{2a-1}}
      -iI_{\bf{a}}I_k\frac{\partial}{\partial x_{2a-1}}\right).
  \end{equation}
are well defined at a point $(x,\bf{a}) \in Z$.  These vector fields can also be identified as
  \begin{eqnarray}
  W_0^a=\lambda_1V_1^a+\lambda_2V_2^a,
  & &
  W_1^a=i(\lambda_1V_1^a-\lambda_2V_2^a),
  \nonumber\\
  W_2^a=\lambda_1V_2^a-\lambda_2V_1^a,
   & &
  W_3^a=i(\lambda_1V_2^a+\lambda_2V_1^a).
  \end{eqnarray}

Clearly, $W_i^a$ are global holomorphic vector fields on $Z$ which also commute. In particular if $W=span\{W_k^a\}$, then any nonzero element in $\Lambda^2(W)$ is a holomorphic Poisson structure on $Z$. In particular one can see that $$\lambda_1^2\sum V_1^a\wedge V_2^a, \lambda_1\lambda_2\sum V_1^a\wedge V_2^a, \lambda_2^2\sum V_1^a\wedge V_2^a$$
define holomorphic Poisson structures on $Z$. Moreover we see that the bivector $\mathcal{P}$ from Theorem \ref{L4} is given by $\mathcal{P}=\lambda_1\lambda_2\sum V_1^a\wedge V_2^a$. Finally we notice that all vector and bivector fields descend to the quotients of $\mathbb{R}^{4m}$ by a commutative lattice induced by translations. So they are globally defined also on the torus and its twistor space. As a result of this discussion we obtain:

\begin{theorem}\label{flatcase}
Let $M=T^{4m}$ be endowed with its  flat hyperk\"ahler structure and $\mathcal{P}=\mathcal{P}_{(x,{\bf a})}$ be the bivector defined in (\ref{cxPoisson}) on its twistor space $Z$. Then $\mathcal{P}, \frac{1}{\lambda}\mathcal{P}, \lambda\mathcal{P}$ are globally defined holomorphic Poisson structures on $Z$.

\end{theorem}

In the next section we partially extend this result to the twistor space of arbitrary hyperk\"ahler manifolds.

\section{Holomorphic Poisson structures on twistor spaces of hyperk\"ahler manifolds}

Unlike the complex case, "quaternionic" coordinates like $(w_i^a)$ exist only in the flat case. However, the 2-vector $\mathcal{P}$ is a global complex Poisson structure and doesn't depend on existence of such coordinates. To check whether it is holomorphic we can use the Chern connection defined as the only metric connection for which the $(0,1)$ part coincides with the $\overline{\partial}$ operator on the tangent space. For a Hermitian manifold with metric $g$ and complex structure $J$, this connection is determined by $$g(\nabla^{Ch}_X Y,Z) = g(\nabla^{LC}_X Y,Z) -\frac{1}{2}dF(JX,Y,Z)$$ where $\nabla^{LC}$ is the Levi-Civita connection and $F(X,Y)=g(JX,Y)$ is the fundamental form.

In general, $dF$ for the twistor space contains the component of the curvature of the base manifold $M$. However, in the hyperk\"ahler case it does not, so one expects that the flat and the general curved case will not be different. We confirm this observation with explicit calculations.

\begin{theorem}
If $M$ is hyperk\"ahler, then $\mathcal{P}$ is a holomorphic Poisson structure which vanishes on the two fibers of $\pi:Z\rightarrow \mathbb{C}P^1$ corresponding to $I$ and $-I$ or $\lambda=0,\infty$. The leaves of the symplectic foliation are given by the fibers of $\pi:Z-\{\pi^{-1}(0,\infty)\} \rightarrow \mathbb{C}P^1- \{0,\infty\}$ and the points of $\pi^{-1}(0)$ and $\pi^{-1}(\infty)$. The converse also holds:  if $g$ is a hyperhermitian metric on $M$ and $\mathcal{P}$ is holomorphic Poisson on $Z=Z(M)$, then $g$ is hyperk\"ahler.
\end{theorem}

{\it Proof:} Endow the twistor space $Z=M\times CP^1$ with the product metric $g_Z=(g, g_{FS})$ where $g_{FS}$ is the canonical (Fubini-Studi) metric on $CP^1=S^2$. Then $(g_Z, \mathcal{I})$ is a Hermitian structure and we have to show that $\nabla_X^{0,1} \mathcal{P}=0$ for every $(0,1)$ vector $X^{0,1}$ on $Z$. Consider as before a local quaternionic-hermitian frame $Z_i, W_i$ on $M$. The local frame $$ Z_i^{\lambda} = \frac{1}{\sqrt{1+|\lambda|^2}}(\overline{\lambda}Z_i  - \overline{W_i}), W_i^{\lambda} = \frac{1}{\sqrt{1+|\lambda|^2}}(\overline{\lambda}W_i  + \overline{Z_i}), \frac{\partial}{\partial\lambda}$$ consists of smooth $(1,0)$ vectors on $Z$ which which is orthogonal, but not orthonormal because $\frac{\partial}{\partial\lambda}$ is not normalized. Then $\mathcal{P} = \sum_i Z_i^{\lambda}\wedge W_i^{\lambda}$. We use for $X^{0,1}$ the vectors of the conjugate $(0,1)$ basis of the basis above. We first check that $\nabla^{Ch}_{\frac{\partial}{\partial\overline{\lambda}}}\mathcal{P}=0$. To this end we use that $\nabla^{Ch}_{X^{0,1}}Y^{1,0} = [X^{0,1},Y^{1,0}]^{1,0}$ where superscript $\{1,0\}$ means the $(1,0)$-component and same for $\{0,1\}$. From here we have also that $\nabla^{Ch}_{X^{0,1}}Y^{1,0}\wedge Z^{1,0} = [X^{0,1},Y^{1,0}]^{1,0}\wedge Z^{1,0} + Y^{1,0}\wedge[X^{0,1},Z^{1,0}]^{1,0} = [X^{0,1}, Y^{1,0}\wedge Z^{1,0}]^{2,0}$. Since $Z=M\times\mathbb{C}P^1$ as a smooth manifold, $[\frac{\partial}{\partial\overline{\lambda}}, Z_i]= [\frac{\partial}{\partial\overline{\lambda}},W_j]=0$ and we obtain
\begin{equation}
\left[\frac{\partial}{\partial\overline{\lambda}}, Z_i^{\lambda}\right] = \frac{1}{(1+|\lambda|^2)^2}\left(Z_i+\lambda\overline{W_i}\right) = \frac{1}{1+|\lambda|^2}\overline{W_i^{\lambda}}
\end{equation}
as well as
\begin{equation}
\left[\frac{\partial}{\partial\overline{\lambda}}, W_i^{\lambda}\right] = -\frac{1}{1+|\lambda|^2}\overline{Z_i^{\lambda}}
\end{equation}
and
\begin{equation}
\left[\frac{\partial}{\partial\overline{\lambda}}, \mathcal{P}\right] = -\frac{1}{1+|\lambda|^2} (Z_i^{\lambda}\wedge\overline{Z_i^{\lambda}}+W_i^{\lambda}\wedge\overline{W_i^{\lambda}})
\end{equation}

So
\begin{equation}
\nabla^{Ch}_{\frac{\partial}{\partial\overline{\lambda}}}\mathcal{P} = 0
\end{equation}

Now for the other vectors we need to use the definition and the fact that $\sum Z_i\wedge W_i, \sum Z_i\wedge \overline{Z_i}+W_i\wedge \overline{W_i}$ are parallel with respect to the Levi-Civita connection on $Z$, since they are parallel on $M$ and the metric on $Z$ is the direct product of the metric on $M$ and the Fubini-Studi metric. In particular
\begin{equation}
\nabla^{LC}_{\overline{Z}_i^{\lambda}} \mathcal{P} =\nabla^{LC}_{\overline{W}_j^{\lambda}} \mathcal{P} = 0
\end{equation}
since $Z_i^{\lambda}$ is a linear combination of $Z_i,W_j,\overline{Z_I},\overline{W_i}$ and $\mathcal{P}$ is a linear combination of $\sum Z_i\wedge W_i, \sum Z_i\wedge \overline{Z_i}+W_i\wedge \overline{W_i}$ with coefficients depending only on $\lambda$.

The difference between the Chern and Levi-Civita connection is proportional to the differential of $F_{\lambda} = \sum \sigma_i^{\lambda}\wedge\overline{\sigma_i^{\lambda}}+\delta_i^{\lambda}\wedge\overline{\delta_i^{\lambda}}$.
We use again that $$d\sum(\sigma_i\wedge\overline{\sigma_i}+\delta_i\wedge\overline{\delta_i}) = d\sum\sigma_i\wedge\delta_i = d \sum \overline{\sigma_i}\wedge\overline{\delta_i} = 0$$
to obtain
$$
dF_{\lambda} = \frac{2d\lambda}{(1+|\lambda|^2)^2}\wedge\left( \overline{\lambda}\sum(\sigma_i\wedge\overline{\sigma_i}+\delta_i\wedge\overline{\delta_i}) - \sigma_i\wedge\delta_i - \overline{\lambda}^2\sum \overline{\sigma_i}\wedge\overline{\delta_i}\right) +$$
$$
+\frac{2d\overline{\lambda}}{(1+|\lambda|^2)^2}\wedge\left(\lambda\sum(\sigma_i\wedge\overline{\sigma_i}+\delta_i\wedge\overline{\delta_i}) + {\lambda}^2\sum\sigma_i\wedge\delta_i +\sum \overline{\sigma_i}\wedge\overline{\delta_i}\right)
$$
and after substitution we get
\begin{equation}
dF_{\lambda} = \frac{2d\lambda}{1+|\lambda|^2}\left(\sum\overline{\sigma_i^{\lambda}}\wedge\overline{\delta_i^{\lambda}}\right)-\frac{2d\overline{\lambda}}{1+|\lambda|^2}
\left(\sum\sigma_i^{\lambda}\wedge\delta_i^{\lambda}\right)
\end{equation}
From here we see that
\begin{equation}
g(\nabla^{Ch}_{\overline{Z}_i^{\lambda}}W_j^{\lambda}-\nabla^{LC}_{\overline{Z}_i^{\lambda}}W_j^{\lambda}, X) = 0
\end{equation}
and similarly for $\overline{W_i^{\lambda}}$. Hence

\begin{equation}
\nabla^{Ch}_{\overline{Z}_i^{\lambda}} \mathcal{P} =\nabla^{Ch}_{\overline{W}_j^{\lambda}} \mathcal{P} = 0
\end{equation}

which proves that $\mathcal{P}$ is holomorphic. It is of maximal rank on all points of $Z$ except $\pi^{-1}(0)$ and $\pi^{-1}(\infty)$, so the Theorem follows.

{\it Q.E.D.}

\begin{rem}
In the the local basis $Z_i=X_i-iIX_i, W_i = JX_i-iKX_i$, the local vector fields from (\ref{twisted fields}) given as $W_k^i=\frac{1}{2}(I_kX_i -iI_{\bf a}I_kX_i)$ again are well defined for all ${\bf a} \in S^2$ so $\frac{1}{\lambda}\mathcal{P}, \lambda\mathcal{P}$ also can be expressed via $W_k^i$. As a consequence we obtain that they are also globally defined and holomorphic Poisson as in Theorem \ref{flatcase}.One can see this family also by changing the fixed $I$ to another structure of the hypercomplex family in the definition of $\mathcal{P}$.

In \cite{HKLR} the twistor space $Z$ of a hyperk\"ahler manifold is characterized in terms of the real structure on $Z$ and the twisted holomorphic symplectic form on the fibers of $\pi$. A similar characterization could be found in terms of the holomorphic Poisson structure $\mathcal{P}$ above. We leave the details of this discussion and some applications as future work.

\end{rem}

\section{Relations to hypercomplex and quaternionic deformations}

In this section we relate the holomorphic Poisson structure $\mathcal{P}$ on twistor spaces to deformatios of hypercomplex and quaternionic structures in the case of a torus and $K3$-surface.

Deformations of hypercomplex structures are described in terms of their twistor spaces in \cite{PP2}. Let $Z=X\times S^2$ be the twistor space of the hypercomplex manifold $X$ with its structure $\mathcal{I}$ and holomorphic projection $\pi: Z\to S^2$. Then $\pi$ defines an exact sequence
\begin{equation}\label{hcxdefcomplex}
0 \longrightarrow \mathcal{D} \longrightarrow \Theta \stackrel{d\pi}\longrightarrow \Theta_{CP^1}\longrightarrow 0
\end{equation}
where  $\mathcal{D}= Ker(d\pi)$.
There is also an  anti-holomorphic involution $\tau : Z\to Z$ which covers the anti-podal map $\rho$ on $S^2$ i.e. $\pi\circ \tau = \rho\circ\pi$.

  As explained in \cite{PP2}, deformations of hypercomplex structures
  are identified to deformations of the real map $\pi$.
  The latter are described by the cohomology spaces
  $H^k(Z, \mathcal{D})$  \cite{Hor}. The real part of these spaces contains the deformation
  of the hypercomplex structures, where "real" in these terms refers to the fixed point of the involution $\tau$. More precisely, $k=0$ corresponds to the Lie algebra of infinitesimal hypercomplex automorphisms, $k=1$ corresponds to the infinitesimal deformations of the hypercomplex structure and $k=2$ is the space of obstructions to extending an infinitesimal deformation to an actual one. Since there is also a correspondence
  between quaternionic structures and the complex structures on the twistor space
  \cite{Salamon}, the real part of the cohomology spaces $H^k(Z, \Theta_Z)$
  contains the deformation theory of the quaternionic structures.


When $\omega$ is $(1,1)$-form on $Z$, $\mathcal{P}(\omega)$, given by the natural contraction, is a 1-form with coefficients in the tangent sheaf. Since $\mathcal{P}$ is holomorphic, this contraction defines a map $\mathcal{P}: H^{(1,1)}(Z,\mathbb{R})\rightarrow H^1(Z,\Theta_Z)$.
In our case, however, $\mathcal{P}$ doesn't have terms depending on $\frac{\partial}{\partial\lambda}$ so it maps $H^{(1,1)}(Z,\mathbb{R})$ into $H^1(Z,\mathcal{D})$, although this map may have a non-trivial kernel.
For the hypercomplex deformations, we need "real" elements of $H^1(Z,D)$. Since $\rho({\bf a})=-{\bf a}$,  $\mathcal{P}$ is real, meaning that $d\tau(\mathcal{P}) = \overline{\mathcal{P}}$. Similarly, the reality condition holds for the quaternionic deformations.

The main question is the integrability of such infinitesimal deformations. The first result in this direction is due to
Bogomolov \cite{Bog}, who proved the integrability of the infinitesimal complex deformations on a K\"ahler holomorphic symplectic manifold $M$. It has been extended in various ways, one of which is for holomorphic K\"ahler Poisson manifolds by Goto (\cite{Goto}). Based on Bogomolov's method, Hitchin \cite{Hitchin} extended Goto's result to the non-k\"ahler case, but still satisfying an additional condition - either $\partial\overline{\partial}$-lemma, or vanishing of $H^2(M,\mathcal{O})=0$. Later the condition was relaxed in \cite{FM} to surjectivity of the natural map $p: H^2(M,\mathbb{R}) \to H^2(M,\mathcal{O})$, but in fact the proof in \cite{Hitchin} covers this case too, since it only uses the following: for every $\overline{\partial}$-closed $(0,2)$-form $\alpha$, there is $(1,1)$-form $\beta$, such that $\partial\alpha = \overline{\partial}\beta$.
As it is proved in \cite{ES}, the Fr\"{o}licher spectral sequence of the twistor space of a $K3$-surface degenerates at $E_1$ level, so the map $p: H^2(M,\mathbb{R}) \to H^2(M,\mathcal{O})$ is surjective. In particular, we have,
\begin{cor}
If $M$ is a K3-surface and $\pi_1:Z\to M$ is the twistor space and $\omega\in H^{(1,1)}(Z,\mathbb{R})$ is represented by $\tau$-invariant form, then $\mathcal{P}(\omega)$ integrates to a complex deformation of $Z$, which induces quaternionic deformation on $M$
\end{cor}

The definition of $\mathcal{P}$ depends on the choice of a fixed complex structure $I$.
Note that by changing $I$ we obtain a 3-dimensional family of real holomorphic Poisson bivectors on $Z$. It is the
real part of the family generated by the bivectors in Theorem 4.1. On the other side the space $$H_-(M) = \{[\omega]\in H^2(M, \mathbb{R})|\omega(IX,IY)=\omega(JX,JY)=\omega(X,Y)$$ is mapped into $H^{(1,1)}(Z,\mathbb{R})$ via $\pi_1^*$. Now, a naive dimension count suggests that the dimension of quaternionic deformations obtained in this way is at most $3(b_2(M) -3)$, which is the same as the dimension of the space of twistor lines in the period domain of $M$. To assert that every twistor line is obtained in this way, however, one also needs $\mathcal{P}$ to be injective.

\vspace{.1in}

In the case of tori, we can explicitly determine the image $\mathcal{P}(H^{(1,1)}(Z,\mathbb{R}))$ and its relations to the deformations.
Following \cite{GPP} define
\begin{equation}\label{omegas}
  \Oomega_1^a=\frac
  {{\overline\lambda}_1d{\overline z}_1^a-{\overline\lambda}_2dz_2^a}
  {|\lambda_1|^2+|\lambda_2|^2},
  \hspace{.2in}
  \Oomega_2^a=\frac
  {{\overline\lambda}_1d{\overline z}_2^a+{\overline\lambda}_2dz_1^a}
  {|\lambda_1|^2+|\lambda_2|^2}.
  \end{equation}
  Then $\{\lambda_1\Oomega_1^a, \lambda_2\Oomega_1^a,
  \lambda_1\Oomega_2^a, \lambda_2\Oomega_2^a\}$ forms a basis for
   the space $H^1(Z, \co )$.
Using the notation from Section 4 and results from \cite{GPP} (see the discussion after Lemma 3), the elements
  $\lambda_1^{1-\ell}\lambda_2^\ell V_i^a\otimes\Omega_j^b$, with
  $0\leq \ell\leq 1$, $1\leq i,j\leq 2$, $1\leq a,b\leq m$
  form a basis for $H^1(Z, \mathcal{D})$. Also the induced cohomology sequence of the complex \ref{hcxdefcomplex} has the first coboundary map $\delta_0$ injective. So the sequence:
   \begin{equation}\label{surjection}
  0\to H^0(Z, p^*\co(2))\to H^1(Z, \mathcal{D})\to H^1(Z, \Theta_Z)\to 0.
  \end{equation}
  is exact and $H^0(Z, p^*\co(2))$ is 3-dimensional. From \cite{GPP} we know that every real element in $H^1(Z, \mathcal{D})$ is tangent to an actual deformation inducing a hypercomplex deformation. Moreover, every real element of $H^1(Z, \Theta_Z)$ defines a quaternionic deformation and every quaternionic deformation of the canonical hypercomplex structure leads to a hypercomplex structure. Also, for every hypercomplex deformation parameter, there is a 3-dimensional hypercomplex deformation within one quaternionic class. Since $dim(H^1(Z,\mathcal{D})=12m^2$ when the torus has dimension $4m$, the local deformation space for hypercomplex deformations is $12m^2$-dimensional and the quaternionic deformations have dimension $12m^2-3$. On the other hand, $H_-(T^{4m})$ has dimension $2m^2+m$ and we have the following:
  \begin{theorem}
  If $\pi_1: Z\to T^{4m}$ is the twistor space of the $4m$-tori, then the map $\mathcal{P}: \pi_1^*(H_-)\to H^1(Z, \mathcal{D})$ is injective for every choice of $\mathcal{P}$ in the 3-dimensional family.
  \end{theorem}

{\it Proof:} We first note that the contraction $\mathcal{P}(\omega)$ can be considered as a map from the tangent bundle to itself given by the composition $T\stackrel{\omega}\to T^*\stackrel{\mathcal{P}}\to T$  and is determined by:
$$X\wedge Y(\alpha\wedge\beta) = \alpha(Y)\beta\otimes X-\alpha(X)\beta\otimes Y-\beta(X)\otimes Y + \beta(Y)\alpha\otimes X$$
When $\mathcal{P} = \sum_a\frac{\partial}{\partial w_1^a}\wedge \frac{\partial}{\partial w_2^a}$ we obtain
$$\mathcal{P}(dz_1^a\wedge d\overline{z}_1^b)=\frac{\overline{\lambda}_1}{|\lambda_1|^2+|\lambda_2|^2}V_2^a\otimes d\overline{z}_1^b+\frac{\overline{\lambda}_2}{|\lambda_1|^2+|\lambda_2|^2}V_1^b\otimes d\overline{z}_1^a$$

$$\mathcal{P}(dz_1^a\wedge d\overline{z}_2^b)=\frac{\overline{\lambda}_1}{|\lambda_1|^2+|\lambda_2|^2}V_2^a\otimes d\overline{z}_2^b+\frac{\overline{\lambda}_2}{|\lambda_1|^2+|\lambda_2|^2}V_2^b\otimes d\overline{z}_1^a$$

$$\mathcal{P}(dz_2^a\wedge d\overline{z}_1^b)=-\frac{\overline{\lambda}_1}{|\lambda_1|^2+|\lambda_2|^2}V_1^a\otimes d\overline{z}_1^b+\frac{\overline{\lambda}_2}{|\lambda_1|^2+|\lambda_2|^2}V_2^b\otimes d\overline{z}_2^a$$

$$\mathcal{P}(dz_2^a\wedge d\overline{z}_2^b)=\frac{\overline{\lambda}_1}{|\lambda_1|^2+|\lambda_2|^2}V_1^a\otimes d\overline{z}_2^b+\frac{\overline{\lambda}_2}{|\lambda_1|^2+|\lambda_2|^2}V_2^b\otimes d\overline{z}_2^a$$

Now we notice that $H_-(T^{4m})$ is spanned by $dz_1^a\wedge d\overline{z}_1^b+d\overline{z}_2^a\wedge dz_2^b, dz_1^a\wedge d\overline{z}_2^b-d\overline{z}_2^a\wedge dz_1^b, dz_2^a\wedge d\overline{z}_1^b-d\overline{z}_1^a\wedge dz_2^b$. From the definition of $\overline{\Omega}_i^a$ and the calculations above we obtain
$$\mathcal{P}(dz_1^a\wedge d\overline{z}_1^b+d\overline{z}_2^a\wedge dz_2^b) = V_2^a\otimes\overline{\Omega}_1^b+ V_1^b\otimes\overline{\Omega}_2^a$$

$$\mathcal{P}(dz_1^a\wedge d\overline{z}_2^b-d\overline{z}_2^a\wedge dz_1^b) = V_2^a\otimes\overline{\Omega}_2^b+ V_2^b\otimes\overline{\Omega}_2^a$$

$$\mathcal{P}(dz_2^a\wedge d\overline{z}_1^b-d\overline{z}_1^a\wedge dz_2^b) = -V_1^a\otimes\overline{\Omega}_1^b- V_1^b\otimes\overline{\Omega}_1^a$$

From here the Theorem follows.

{\it Q.E.D.}


\begin{thebibliography}{12}

\bibitem{BS} B. Banos, A. Swann, {\em Potentials for hyper-K\"ahler metrics with torsion,}
Classical Quantum gravity 21 (2004), no. 13, 3127–-3135.
\smallskip

\bibitem{BGP} L. Bedulli, A. Gori, F. Podesta {\it Homogeneous hyper-complex structures and Joyce's construction}, Diff.Geom.Appl. {\bf 29} (2011) 547 -- 554.
\smallskip

\bibitem{Bes}
A. Besse, {\em Einstein Manifolds}, Springer-Verlag, New York (1987)
\smallskip


\bibitem{Bog} F. Bogomolov, {\it Hamiltonian K\"ahlerian spaces}, Dokl. Akad.Nauk SSSR {\bf 243}(1978) no 6, 1462 - 1465.
\smallskip

\bibitem{CFIU} L. Cordero, M. Fernandez, R. Ibanez, L. Ugarte,{\em Almost complex Poisson manifolds}, Ann.Glob.An.Geom. {\bf 18} 265-290 (2000)
\smallskip


\bibitem{VaskoGosho} G. Dimitrov, V. Tsanov {\it Homogeneous Hypercomplex Structures I - the compact Lie groups}, preprint  arXiv:1005.0172.
\smallskip


\bibitem{ES} M. Eastwood, M. Singer, {\it The Fr\"ohlicher spectral sequence on a twistor space}, J.Diff. Geom. {\bf 38} (1993) 653 - 669.
\smallskip

\bibitem{FM} D. Fiorenza, M. Manetti, {\it  Formality of Koszul brackets and deformations of holomorphic Poisson manifolds} Homology Homotopy Appl. {\bf 14} (2012), no. 2, 63 –- 75.
\smallskip

\bibitem{Goto} R. Goto, {\it Deformations of generalized complex and generalized K\"ahler structures}, J.Diff.Geom. {\bf 84} (2010) no 3, 525 - 560.

\bibitem{GP} G. Grantcharov, Y.S. Poon, {\em Geometry of
  hyper-K\"ahler connections with torsion},
Comm. Math. Phys. 213 (2000), no. 1, 19--37.
\smallskip

\bibitem{GPP}  G. Grantcharov, H. Pedersen, Y.S. Poon, {\em Deformations of hypercomplex structures associated to Heisenberg groups}, Quart.J.Math {\bf 59} (2008), 335-362.
\smallskip

\bibitem{GLV} G. Grantcharov, M. Lejmi, M. Verbitsky, {\em Existence of HKT metrics on hypercomplex manifolds of real dimension 8}, preprint arXiv:1409.3280.
\smallskip

\bibitem{HKLR} N. Hitchin, A. Karlhede, U. Lindstr\"{o}m, M. Ro\v{c}ek, {\em Hyperk\"ahler metrics and supersymmetry},  Comm. Math. Phys. {\bf 108} (1987) 535-589.
\smallskip

\bibitem{Hitchin} N. Hitchin, {\it Deformations of holomorphic Poisson manifolds} Moscow Math. J. {\bf 12} (2012) no 3 567 - 591.
\smallskip
 \bibitem{Hor} E. Horikawa.
  {\em On deformations of holomorphic maps I},
  J. Math. Soc. Japan, {\bf 25}
  (1973) 372--396.
  {\em II},
  J. Math. Soc. Japan, {\bf 26}
  (1974) 647--667.


\bibitem{HP} P.S. Howe, G. Papadopoulos,  {\em Twistor spaces for
  hyper-K\"ahler manifolds with torsion} Phys. Lett. B 379
(1996), no. 1-4, 80--86.
\smallskip


 \bibitem{Joyce2} D. Joyce.
  {\em Compact hypercomplex and quaternionic manifolds},
  J. Differential Geom.
  {\bf 35} (1992) 743--761.
  \smallskip

\bibitem{L} A. Lichnerowicz, {\em Verietes de Jacobi et espaces homogenes de contact complexes}, J.Math Pures Appl. {\bf 67}(9) 131-173 (1988)
\smallskip


\bibitem{P} A. Panasyuk, {\em Isomorphisms of some complex Poison brackets} Ann.Glob.An.Geom.{\bf 15} 313-324 (1997).
\smallskip

\bibitem{PP2} H. Pedersen \& Y. S. Poon.
  {\em Deformations of hypercomplex structures},
  J. reine angew. Math. {\bf 499}
  (1998) 81-99.
\smallskip

\bibitem{Salamon} S. M. Salamon.
  {\em Differential geometry of quaternionic manifolds},
  Ann. scient. \'{E}c. Norm. Sup. $4^e$,
  {\bf 19} (1986) 31--55.
\smallskip

\bibitem{SSTV} P. Spindel, A. Servin, W. Troost, A. Van Proeyen {\it Extended supersymmetric $\sigma$-models on group manifolds}Nucl. Phys. B, {\bf B308} (1988) 662 -- 698.
\smallskip

\bibitem{V} M. Verbitsky, {\em Hyperk\"ahler manifolds with torsion,
  supersymmetry and Hodge theory}, math.AG/0112215,
Asian J. Math. Vol. 6, No. 4, pp. 679-712 (2002).



\end{thebibliography}
\end{document}